\documentclass{amsart}
\usepackage{amsmath}
\usepackage{latexsym}
\usepackage{color}
\usepackage{amscd}
\usepackage[all]{xy}
\usepackage{enumerate}
\usepackage{mathrsfs}
\usepackage
{hyperref}
\hypersetup{
setpagesize=false,
 bookmarksnumbered=true,%
 bookmarksopen=true,%
 colorlinks=true,%
 linkcolor=blue,
 citecolor=red,
}
\usepackage{soul}
\usepackage{graphicx}
\usepackage{comment}
\newtheorem{thm}{Theorem}[section]
\newtheorem{lem}[thm]{Lemma}
\newtheorem{cor}[thm]{Corollary}
\newtheorem{prop}[thm]{Proposition}

\newtheorem{rem}[thm]{Remark}

\theoremstyle{definition}

\def\M{\mathcal{M}_{g}}

\def\Sg{\Sigma_{g}}

\begin{document}

\title[Nielsen equivalence classes of mapping class groups]
{On Nielsen equivalence classes of two-elements generators of mapping class groups}

\author[S. Hirose]{Susumu Hirose}
\address{Department of Mathematics, Faculty of Science and Technology, Tokyo University of Science, Noda, Chiba 278-8510, Japan}
\email{hirose-susumu@rs.tus.ac.jp}

\author[N. Monden]{Naoyuki Monden}
\address{Department of Mathematics, Faculty of Science, Okayama University, Okayama 700-8530, Japan}
\email{n-monden@okayama-u.ac.jp}

\begin{abstract}
We show that there are infinitely many Nielsen equivalence classes of the mapping class group of a closed oriented surface of genus at least eight. 
\end{abstract}

\maketitle

\section{Introduction}
Let $\M$ be the mapping class group of a closed oriented surface $\Sg$ of genus $g$, i.e., the group of isotopy classes of orientation-preserving homeomorphisms of $\Sg$. 
Note that $\mathcal{M}_1$ is isomorphic to the special linear group $\mathrm{SL}(2,\mathbb{Z})$. 
In this paper, a pair $(x_1,x_2)$ of elements of a group $G$ such that $x_1,x_2$ generate $G$ is called a \textit{generating pair} of $G$. 

The problem of finding generating sets for $\M$ is a classical one, and has been studied by many authors. 
It was shown by Dehn \cite{De} that $\M$ is finitely generated, and his generating set consists of finitely many Dehn twists. 
After that, Lickorish \cite{Li} showed that $3g-1$ Dehn twists suffice to generate $\M$. 
Moreover, Humphries \cite{Hu} reduced the number of Dehn twists generating $\M$ to $2g+1$ and proved that the number is minimal to generate $\M$ if $g\geq 2$. 
If one is not limited to Dehn twists, smaller generating sets can be obtained. 
In fact, Wajnryb \cite{Wa} gave a generating set for $\M$ consisting of two elements for $g\geq 1$. 
Note that this generating set is minimal since $\M$ is not cyclic, that is, at least two elements are needed to generate $\M$.  
Since then, various generating pairs for $\M$ have been constructed (see, for example, \cite{Ko1, Yi2, BK, HM}). 
In particular, infinitely many distinct generating pairs were given in \cite{HM} for large $g$. 
For initiating our investigation on the set of generating pairs for $\M$, we consider a certain equivalence relation among them.

In general, there is a natural equivalence relation on generating sets for a group $G$, called \textit{Nielsen equivalence}.   
An \textit{elementary Nielsen transformation} on a $k$-tuple $\mathcal{X} = (x_1,x_2,\ldots,x_k)$ of elements of $G$ is one of the following three operations: 
\begin{enumerate}
\item[(1)] Replace $x_i$ by $x_i^{-1}$ for some $i \in \{1,2,\ldots,k\}$, 
\item[(2)] Swap $x_i$ and $x_j$ for some $i\neq j$ and $i,j \in \{1,2,\ldots,k\}$, and 
\item[(3)] Replace $x_i$ by $x_ix_j$ for some $i\neq j$ and $i,j \in \{1,2,\ldots,k\}$. 
\end{enumerate}
We say that two $k$-tuples $\mathcal{X}$ and $\mathcal{X}'$ are \textit{Nielsen equivalent} if one can be transformed into the other by a finite sequence of elementary Nielsen transformations. 
Examples of groups with infinitely many Nielsen equivalence classes can be found among certain knot groups (see \cite{Zie,HW}), one-relator groups (see \cite{Bru}), relatively free polynilpotent groups (see, for example, \cite{MN}), the Gupta-Sidki $p$-group for prime $p\geq 3$ (see \cite{Myr}) and others. 
In this paper, we investigate on how many Nielsen equivalence classes of generating pairs exist for the mapping class group $\M$. 
\begin{thm}\label{thm:1}
For each $g\geq 8$, there are infinitely many Nielsen equivalence classes on generating pairs of the mapping class group $\M$.  
\end{thm}
Since there is a surjective homomorphism from $\M$ onto the integral symplectic group $\mathrm{Sp}(2g, \mathbb{Z})$, we obtain the following consequence. 
\begin{cor}\label{cor:1}
For each $g\geq 8$, there are infinitely many Nielsen equivalence classes on generating pairs of $\mathrm{Sp}(2g,\mathbb{Z})$.  
\end{cor}

In contrast to Theorem~\ref{thm:1} and Corollary~\ref{cor:1}, the following holds for $g=1$. 
\begin{prop}\label{prop:100}
There are only finitely many Nielsen equivalence classes on generating pairs of the mapping class group $\mathcal{M}_1(\cong \mathrm{SL}(2,\mathbb{Z}))$. 
\end{prop}

We present a more natural equivalence relation on generating sets for a group $G$, called \textit{T-equivalence}. 
Let $F_k$ be a free group of rank $k$ generated by $f_1,f_2,\ldots,f_k$. 
For a $k$-tuple $\mathcal{X} = (x_1,x_2,\ldots,x_k)$ of elements of $G$ such that $x_1,x_2,\ldots,x_k$ generate $G$, we define an surjective homomorphism $q_\mathcal{X} : F_k \to G$ to be $q_\mathcal{X}(f_i) = x_i$ for $i=1,2,\ldots,k$. 
We say that two $k$-tuples $\mathcal{X}$ and $\mathcal{X}'$ are \textit{T-equivalent} if there are automorphisms $\Phi:F_k \to F_k$ and $\phi:G \to G$ such that $\phi \circ q_{\mathcal{X}} = q_{\mathcal{X}'} \circ \Phi$. 
Since elementary Nielsen transformations on $(f_1,f_2,\ldots,f_k)$ generate the automorphism group $Aut(F_k)$ of $F_k$ by the result of Nilesen \cite{Ni}, 
T-equivalent $\mathcal{X}$, $\mathcal{X}'$ are Nielsen equivalent if $\phi = \mathrm{id}_G$. 
Makoto Sakuma kindly pointed out to us that our generating pairs given in this paper are not pairwise T-equivalent. 
That is, the following holds. 
\begin{thm}\label{thm:1000}
For each $g\geq 8$, there are infinitely many T-equivalence classes on generating pairs of the mapping class group $\M$.  
\end{thm}
The outline of the paper is as follows. 
In Section~\ref{Pre}, we present some results on mapping class groups and Nielsen transformations. 
The proofs of Theorems~\ref{thm:1} and~\ref{thm:1000} and Proposition~\ref{prop:100} are given in Section~\ref{proof}.

\vspace{0.1in}
\noindent \textit{Acknowledgements.} 
The first author was supported by JSPS KAKENHI Grant Numbers JP24K06746. 
The second author was supported by JSPS KAKENHI Grant Numbers JP25K07003. 
The authors would like to thank Marco Linton and Makoto Sakuma for inspiring us to think about this problem. 
They wish to express their gratitude to Makoto Sakuma for his comments on an earlier version of this paper.

\section{Preliminaries}\label{Pre}
This section gives some facts about mapping class groups and Nielsen transformations. 
More details of mapping class groups can be found in \cite{FM}. 
Let us denote by $\M$ the mapping class group of the closed oriented surface $\Sg$ of genus $g$. 
Throughout this paper, we will use the same symbol for a diffeomorphism and its isotopy class, or a simple closed curve and its isotopy class. 
The Dehn twist about a simple closed curve $c$ on $\Sg$ is denoted by $t_c$. 
For $f_1$ and $f_2$ in $\M$, the notation $f_2f_1$ means that we first apply $f_1$ and then $f_2$.

In this paper, we will use the following three relations in $\M$ repeatedly. 
Let $a,b,c,d$ be simple closed curves on $\Sg$. 
\begin{itemize}
\item $t_{f(a)} = f t_a f^{-1}$ for any $f \in \M$.

\item Suppose that $a$ and $b$ are disjoint from each other. 
Then, $t_a t_b = t_b t_a$, $t_b(a) = a$ and $t_a(b) = b$.

\item Suppose that $a$ intersects $b$ transversely at exactly one point. 
Then, $t_at_bt_a = t_bt_at_b$, $t_bt_a(b) = a$ and $t_b^{-1}t_a^{-1}(b) = a$.

\item The \textit{lantern relation}: $t_d t_c t_b t_a = t_z t_y t_x$, where $x,y,z$ are the interior curves on a subsurface of genus $0$ with four boundary curves $a,b,c,d$ on $\Sg$ as in Figure~\ref{lanterncurves}. 
\end{itemize}
\begin{figure}[hbt]
  \centering
       \includegraphics[scale=.07]{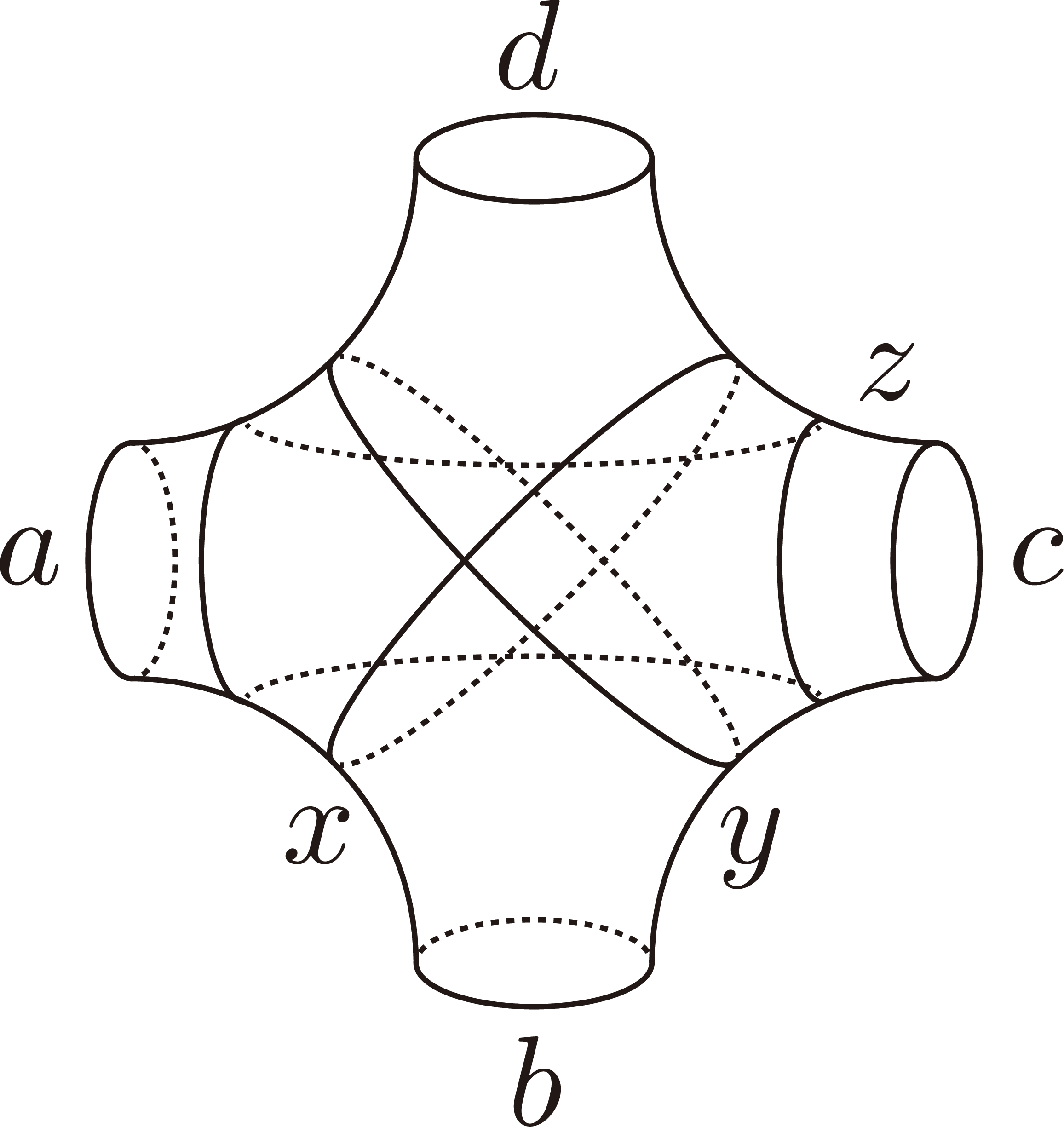}
       \caption{Simple closed curves $a,b,c,d,x,y,z$.} 
       \label{lanterncurves}
  \end{figure}
The lantern relation was discovered by Dehn \cite{De} and rediscovered by Johnson \cite{Jo1}.

We assume that the surface of this paper is the $yz$-plane and that $\Sigma_g$ is embedded in $\mathbb{R}^3$ as in Figure~\ref{rotation} such that it is invariant under the rotation $r$ by $-\frac{2\pi}{g}$ about the $x$-axis. 
The notations $a_i,b_i,c_i$ will always denote the simple closed curves on $\Sigma_g$ shown in Figure~\ref{rotation} for $i=1,2,\ldots,g$. 
\begin{figure}[hbt]
  \centering
       \includegraphics[scale=.17]{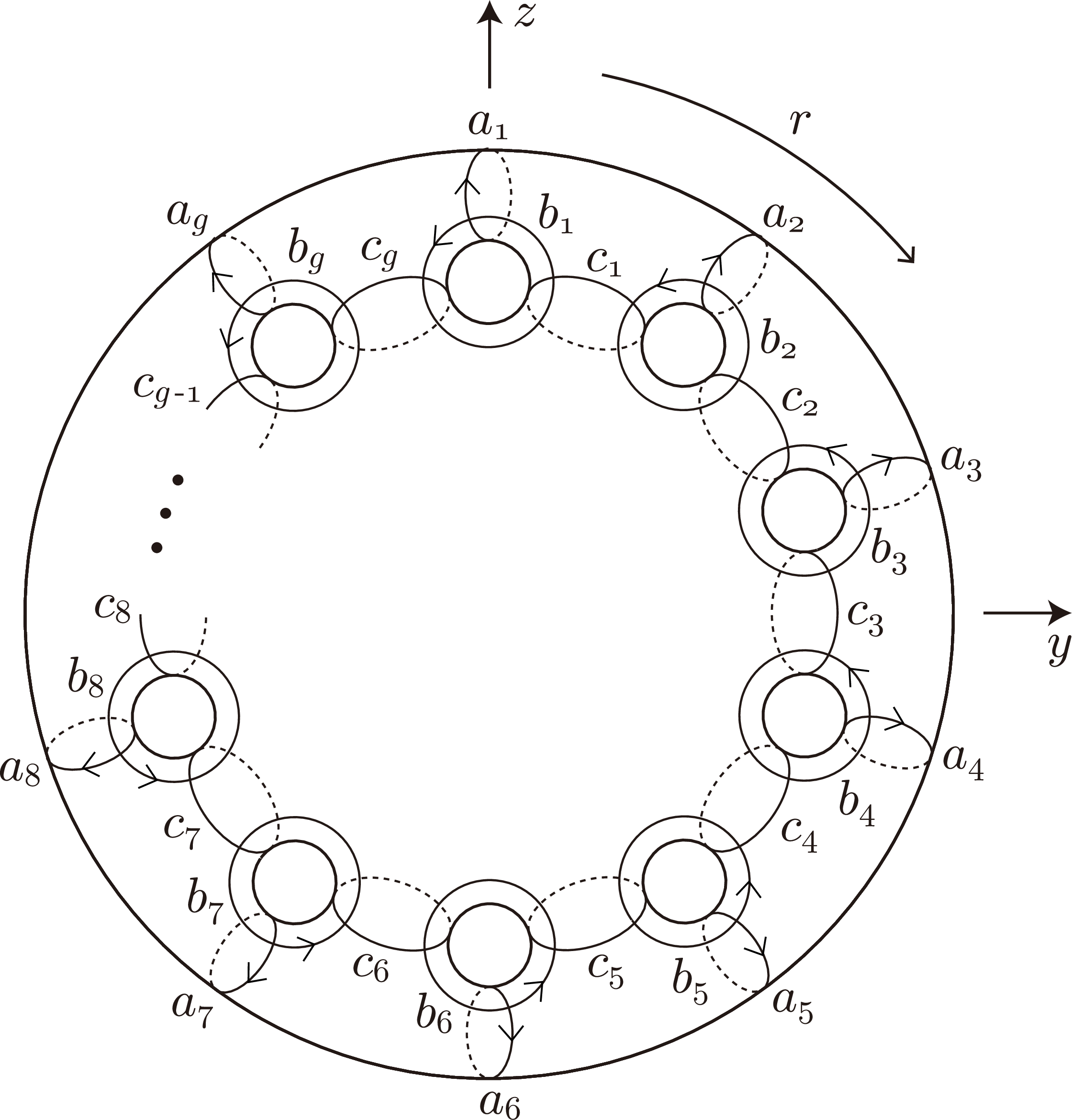}
       \caption{The rotation $r : \Sigma_g \to \Sigma_g$ by $-\frac{2\pi}{g}$ about the $x$-axis and the simple closed curves $a_i,b_i,c_i$ on $\Sigma_g$ for $i=1,2,\ldots,g$.}
       \label{rotation}
  \end{figure}

\begin{thm}[\cite{Li}]\label{Lickorishthm}
For $g\geq 1$, $\M$ is generated by $t_{a_1}, t_{a_2},\ldots,t_{a_g}, t_{b_1},t_{b_2},\ldots,t_{b_g}$ and $t_{c_1},t_{c_2},\ldots,t_{c_{g-1}}$. 
\end{thm}

\begin{lem}[\cite{Ni}]\label{commutator}
Let $(x_1,x_2)$ and $(y_1,y_2)$ be two Nielsen equivalent generating pairs of a group $G$. 
Then the commutator $[x_1,x_2] = x_1x_2x_1^{-1}x_2^{-1}$ is conjugate either to $[y_1,y_2]$ or to $[y_1,y_2]^{-1}$. 
\end{lem}
\begin{proof}
If we apply an elementary Nielsen transformation to a generating pair $(x_1,x_2)$, then the resulting generating pair is either $(x_1^{-1},x_2)$, $(x_1,x_2^{-1})$, $(x_2,x_1)$, $(x_1x_2,x_2)$ or $(x_1,x_2x_1)$. 
Then, we see that  
\begin{align*}
&[x_1^{-1},x_2] = x_1^{-1}[x_1,x_2]^{-1}x_1,& &[x_1,x_2^{-1}] = x_2^{-1}[x_1,x_2]^{-1}x_2,& &[x_2,x_1] = [x_1,x_2]^{-1},& \\
&[x_1x_2,x_2] = [x_1,x_2],& &\mathrm{and}& &[x_1,x_2x_1] = [x_1,x_2],& 
\end{align*}
and the lemma follows. 
\end{proof}

\begin{thm}[\cite{G}, Grushko Theorem]\label{thm:G}
Let $G$ be the free product $G_1\ast G_2$ of two finitely generated groups $G_1$ and $G_2$. 
If an $n$-tuple $\mathcal{X} = (x_1,x_2,\ldots,x_n)$ of elements of $G$ with $G=\langle x_1,x_2,\ldots,x_n \rangle$, then $X$ is Nielsen equivalent to an $n$-tuple $\mathcal{Y} = (y_1,y_2,\ldots,y_k,z_1,z_2,\ldots,z_{n-k})$ with $G_1=\langle y_1,y_2,\ldots,y_k \rangle$ and $G_2= \langle z_1,z_2,\ldots,z_{n-k} \rangle$.  
\end{thm}

The following lemma is a weak version of Theorem 2 in \cite{MP}. 
\begin{lem}\label{lem:1000}
Let $G$ be the free product $G_1\ast_{G_0} G_2$ of two finitely generated groups $G_1$ and $G_2$ with amalgamated subgroup $G_0$, where $G_0$ is a normal subgroup. 
If an $n$-tuple $\mathcal{X} = (x_1,x_2,\ldots,x_k)$ of elements of $G$ with $G = \langle x_1,x_2,\ldots,x_n \rangle$, $X$ is Nielsen equivalent to an $n$-tuple $\mathcal{Y} = (y_1,y_2,\ldots,y_k,z_1,z_2,\ldots,z_{n-k})$ with $y_i \in G_1$ and $z_i \in G_2$ for any $i$. 
\end{lem}
\begin{proof}
Since $G_0$ is a normal subgroup, $G/G_0$ is isomorphic to the free product $G_1/G_0 \ast G_2/G_0$, denoted by $G'$. 
This gives the natural homomorphism $\theta: G \to G'$, and hence we have $G' = \langle \theta(x_1),\theta(x_2),\ldots,\theta(x_n) \rangle$. 
By Theorem~\ref{thm:G}, we see that the $n$-tuple $(\theta(x_1),\theta(x_2),\ldots,\theta(x_n))$ is Nielsen equivalent to an $n$-tuple $\mathcal{Y}' = (y_1',y_2',\ldots,y_k', z_1',z_2',\ldots,z_{n-k}')$ such that $G_1/G_0 = \langle y_1',y_2',\ldots,y_k' \rangle$ and $G_2/G_0 = \langle z_1',z_2',\ldots,z_{n-k}' \rangle$. 
Therefore, by applying the same sequence of elementary Nielsen transformations from $\mathcal{X}'$ to $\mathcal{Y}'$, 
we convert $\mathcal{X}$ into an $n$-tuple $\mathcal{Y} = (y_1,y_2,\ldots,y_k,$ $z_1,z_2,\ldots,z_{n-k})$, where $y_i \in G_1$ and $z_i \in G_2$ for any $i$. 
This finishes the proof. 
\end{proof}

\section{Proof of Theorems~\ref{thm:1} and~\ref{thm:1000} and Proposition~\ref{prop:100}}\label{proof}
Theorem~\ref{thm:1} follows from Theorem~\ref{thm:2} and Proposition~\ref{prop:3} below. 
For abbreviation, we write $\rho_n$ and $h_0$ instead of $rt_{a_1}^n$ and $t_{b_1}t_{b_2}t_{a_3}t_{c_5}$, respectively. 
\begin{thm}\label{thm:2}
Let $g\geq 8$, and let $G$ be the subgroup of $\M$ generated by $\rho_n = rt_{a_1}^n$ and $h_0 = t_{b_1}t_{b_2}t_{a_3}t_{c_5}$. 
Then, $G = \M$ for any integer $n$. 
\end{thm}

Lemma~\ref{lem:7} below is used to show Theorem~\ref{thm:2}. 
\begin{lem}\label{lem:7}
Let $g\geq k+2$, where $k$ is a positive integer, and let $f$ be an element in $\M$ satisfying $f (a_k, b_k, c_k) = (a_{k+1}, b_{k+1}, c_{k+1})$ and $f(a_{k+1}, b_{k+1}) = (a_{k+2}, b_{k+2})$. 
We denote by $G'$ the subgroup of $\M$ generated by $f$, $t_{a_k} t_{a_{k+1}}^{-1}$, $t_{b_k} t_{b_{k+1}}^{-1}$ and $t_{c_k} t_{c_{k+1}}^{-1}$. 
Then, the elements $t_{a_k}, t_{b_k}, t_{c_k}$ are in $G'$.
\end{lem}
\begin{proof}
From the assumptions that $f, t_{a_k} t_{a_{k+1}}^{-1}, t_{b_k} t_{b_{k+1}}^{-1} \in G'$ and $f (a_k, b_k, c_k) = (a_{k+1}, b_{k+1}, c_{k+1})$, the following holds: 
\begin{align*}
&t_{a_{k+1}} t_{a_{k+2}}^{-1} = t_{f(a_{k})} t_{f(a_{k+1})}^{-1} = f t_{a_k} t_{a_{k+1}}^{-1} f^{-1} \in G', \\
&t_{b_{k+1}} t_{b_{k+2}}^{-1} = t_{f(b_{k})} t_{f(b_{k+1})}^{-1} = f t_{b_k} t_{b_{k+1}}^{-1} f^{-1} \in G', \ \mathrm{and} \\
&t_{a_k} t_{a_{k+2}}^{-1} = t_{a_k} t_{a_{k+1}}^{-1} \cdot t_{a_{k+1}} t_{a_{k+2}}^{-1} \in G'. 
\end{align*}
Here, we set $f_1 = t_{a_{k}} t_{a_{k+1}}^{-1} \cdot t_{b_{k}} t_{b_{k+1}}^{-1}$, and hence $f_1$ is in $G'$ by the assumption. 
Since $a_{k+1}$ is disjoint from $b_{k}$, we obtain $f_1 = t_{a_{k}} t_{b_{k}} t_{a_{k+1}}^{-1} t_{b_{k+1}}^{-1}$. 
Moreover, from the fact that $a_i,b_i,c_i$ are disjoint from $a_j,b_j,c_j$ if $|i-j|>1$ and that $a_i$ intersects $b_i$ transversely at exactly one point, we have 
\begin{align*}
&f_1 (a_{k},a_{k+2}) = (t_{a_{k}} t_{b_{k}} (a_{k}), a_{k+2}) = (b_{k}, a_{k+2}), \ \mathrm{and} \\
&t_{b_{k}} t_{a_{k+2}}^{-1} = t_{f_1(a_{k})} t_{f_1(a_{k+2})}^{-1} = f_1 t_{a_k} t_{a_{k+2}}^{-1} f_1^{-1} \in G'. 
\end{align*}
Next, we set $f_2 = t_{b_k} t_{a_{k+2}}^{-1} \cdot t_{c_{k}} t_{c_{k+1}}^{-1}$, and hence $f_2$ is in $G'$ by the assumption. 
Since $a_{k+2}$ is disjoint from $c_{k}$, we obtain $f_2 = t_{b_{k}} t_{c_{k}} t_{a_{k+2}}^{-1} t_{c_{k+1}}^{-1}$. 
By a similar argument to the case of $f_1$, we have
\begin{align*}
&f_2 (b_{k},a_{k+2}) = (t_{b_{k}} t_{c_{k}} (b_{k}), a_{k+2}) = (c_{k}, a_{k+2}), \ \mathrm{and} \\
&t_{c_{k}} t_{a_{k+2}}^{-1} = t_{f_2(b_{k})} t_{f_2(a_{k+2})}^{-1} = f_2 t_{b_k} t_{a_{k+2}}^{-1} f_2^{-1} \in G'. 
\end{align*}
From the argument above, we obtain the following:
\begin{align*}
&t_{a_k} t_{b_{k+1}}^{-1} = t_{a_k} t_{a_{k+2}}^{-1} \cdot (t_{b_{k}} t_{a_{k+2}}^{-1})^{-1} \cdot t_{b_{k}} t_{b_{k+1}}^{-1} \in G', \\
&t_{a_k} t_{c_k}^{-1} = t_{a_k} t_{a_{k+2}}^{-1} \cdot (t_{c_{k}} t_{a_{k+2}}^{-1})^{-1} \in G', \\
&t_{a_k} t_{c_{k+1}}^{-1} = t_{a_k} t_{c_k}^{-1} \cdot t_{c_{k}} t_{c_{k+1}}^{-1} \in G', \ \mathrm{and} \\
&t_{a_k} t_{b_{k+2}}^{-1} = t_{a_k} t_{b_{k+1}}^{-1} \cdot  t_{b_{k+1}} t_{b_{k+2}}^{-1} \in G'. 
\end{align*}
Let $d_1$ and $d_2$ be the simple closed curves in $\Sg$ as in Figure~\ref{curves1}. 
Then, it is easy to check that 
\begin{align*}
 t_{b_{k+1}} t_{a_k}^{-1} \cdot  t_{c_k} t_{a_k}^{-1} \cdot t_{a_k} t_{a_{k+1}}^{-1} \cdot t_{c_{k+1}} t_{a_{k}}^{-1} (b_{k+1}, a_k) &= (d_1, a_k) \\
  t_{b_{k+2}}t_{a_k}^{-1} \cdot t_{c_{k+1}} t_{a_k}^{-1} \cdot t_{a_{k+2}} t_{a_k}^{-1} \cdot t_{b_{k+2}} t_{a_k}^{-1}  (d_1, a_k) &= (d_2, a_k). 
\end{align*}
When we set $\phi_1 = t_{b_{k+1}} t_{a_k}^{-1} \cdot  t_{c_k} t_{a_k}^{-1} \cdot t_{a_k} t_{a_{k+1}}^{-1} \cdot t_{c_{k+1}} t_{a_{k}}^{-1}$ and $\phi_2 = t_{b_{k+2}}t_{a_k}^{-1} \cdot t_{c_{k+1}} t_{a_k}^{-1} \cdot t_{a_{k+2}} t_{a_k}^{-1} \cdot t_{b_{k+2}} t_{a_k}^{-1}$, by $\phi_1, \phi_2, t_{b_{k+1}} t_{a_k}^{-1} \in G'$ we see that 
\begin{align*}
& t_{d_1} t_{a_k}^{-1} = \phi_1 \cdot t_{b_{k+1}} t_{a_k}^{-1} \cdot \phi_1^{-1}\in G', \ \mathrm{and} \\
& t_{d_2} t_{a_k}^{-1} = \phi_2 \cdot t_{d_1} t_{a_k}^{-1} \cdot \phi_2^{-1} \in G'. 
\end{align*}
Moreover, by $t_{a_k} t_{c_k}^{-1} \in G'$, we have $t_{d_2} t_{c_k}^{-1} = t_{d_2} t_{a_k}^{-1} \cdot t_{a_k} t_{c_k}^{-1} \in G'$. 
Therefore, by the lantern relation $t_{a_k} t_{c_k} t_{c_{k+1}} t_{a_{k+2}} = t_{a_{k+1}} t_{d_1} t_{d_2}$ and $t_{a_{k+1}} t_{c_{k+1}}^{-1} = (t_{a_k}t_{a_{k+1}}^{-1})^{-1} \cdot t_{a_k} t_{c_{k+1}}^{-1} \in G'$, we have 
\begin{align*}
t_{a_{k+2}} = t_{a_{k+1}} t_{c_{k+1}}^{-1} \cdot t_{d_1} t_{a_k}^{-1} \cdot t_{d_2} t_{c_k}^{-1} \in G'. 
\end{align*}
By $t_{a_{k}}t_{a_{k+2}}^{-1}, t_{b_k} t_{a_{k+2}}^{-1}, t_{c_k} t_{a_{k+2}}^{-1} \in G'$, the elements $t_{a_k}, t_{b_k}, t_{c_k}$ are contained in $G'$, and the lemma follows. 
\end{proof}
\begin{figure}[hbt]
  \centering
       \includegraphics[scale=.20]{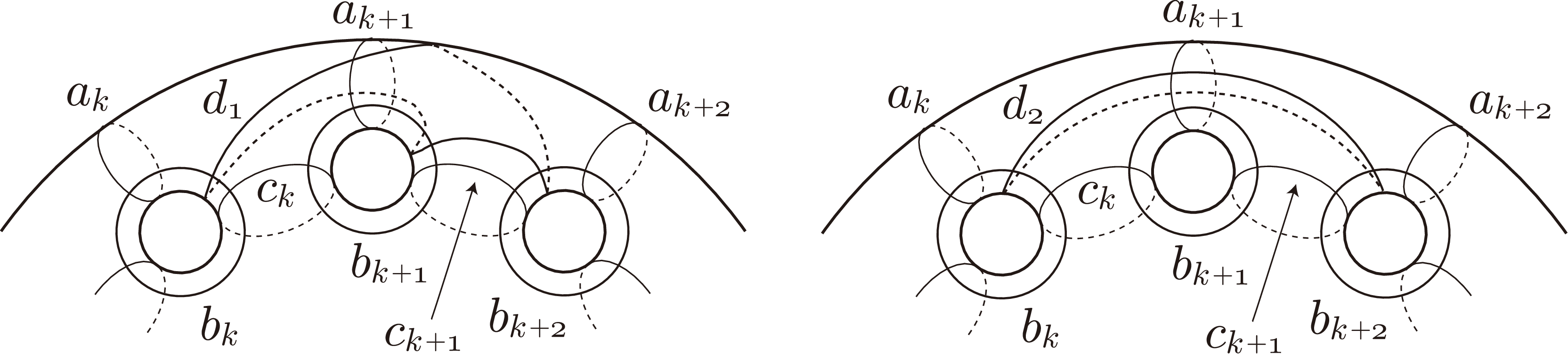}
       \caption{The simple closed curve $d_1$ and $d_2$ on $\Sigma_g$.}
       \label{curves1}
  \end{figure}

\begin{proof}[Proof of Theorem~\ref{thm:2}]
Suppose that $g\geq 8$. 

We set $h_1 = \rho_n h_0 \rho_n^{-1}$ and $h_2 = \rho_n^{-1} h_0 \rho_n$, and hence $h_1$ and $h_2$ are in $G$.  
Then, we see that 
\begin{align*}
h_1 &= \rho_n t_{b_1}t_{b_2}t_{a_3}t_{c_5} \rho_n^{-1} = t_{\rho_n(b_1)}t_{\rho_n(b_2)}t_{\rho_n(a_3)}t_{\rho_n(c_5)}, \ \mathrm{and} \\
h_2 &= \rho_n^{-1} t_{b_1}t_{b_2}t_{a_3}t_{c_5} \rho_n = t_{\rho_n^{-1}(b_1)}t_{\rho_n^{-1}(b_2)}t_{\rho_n^{-1}(a_3)}t_{\rho_n^{-1}(c_5)}.
\end{align*} 
Since $a_1$ is disjoint from $b_g, a_2, b_2, a_3, c_4, c_5$ (see Figure~\ref{curves} (1)), we have
\begin{align*}
&\rho_n(b_1,b_2,a_3,c_5) = r t_{a_1}^n(b_1,b_2,a_3,c_5) = r(t_{a_1}^n(b_1),b_2,a_3,c_5) = (t_{a_2}^n(b_2),b_3,a_4,c_6), \ \mathrm{and}\\
&\rho_n^{-1}(b_1,b_2,a_3,c_5) = t_{a_1}^{-n}r^{-1}(b_1,b_2,a_3,c_5) = t_{a_1}^{-n}(b_g,b_1,a_2,c_4) = (b_g, t_{a_1}^{-n}(b_1),a_2,c_4).  
\end{align*}
Therefore, we have $h_1 = t_{t_{a_2}^n(b_2)} t_{b_3} t_{a_4} t_{c_6}$ and $h_2 = t_{b_g} t_{t_{a_1}^{-n}(b_1)} t_{a_2} t_{c_4}$.

Let $h_3 = h_2^{-n} h_1 h_2^n$, and hence $h_3$ is in $G$ by $h_1,h_2 \in G$. 
Then, we see that 
\begin{align*}
h_3 = h_2^{-n} t_{t_{a_2}^n(b_2)} t_{b_3} t_{a_4} t_{c_6} h_2^n =    t_{h_2^{-n}(t_{a_2}^n(b_2))}   t_{h_2^{-n}(b_3)} t_{h_2^{-n}(a_4)} t_{h_2^{-n}(c_6)}. 
\end{align*}
Since $b_g, t_{a_1}^{-n}(b_1), a_2, c_4$ are disjoint from each other (see Figure~\ref{curves} (2)), we have 
\begin{align*}
h_2^{-n} = (t_{b_g} t_{t_{a_1}^{-n}(b_1)} t_{a_2} t_{c_4})^{-n} = t_{a_2}^{-n} t_{b_g}^{-n} t_{t_{a_1}^{-n}(b_1)}^{-n} t_{c_4}^{-n}.
\end{align*} 
Here, by the assumption that $g\geq 8$,  $c_6$ is disjoint from $b_g$. 
Since $b_g, t_{a_1}^{-n}(b_1), c_4$ (resp. $a_2$) are disjoint from $t_{a_2}^n(b_2),b_3,a_4,c_6$ (resp. $b_3,a_4,c_6$) (see Figure~\ref{curves} (3) (resp. (4))), we obtain 
\begin{align*}
h_2^{-n}(t_{a_2}^n(b_2),b_3,a_4,c_6) = (t_{a_2}^{-n}(t_{a_2}^n(b_2)), b_3, a_4, c_6) = (b_2, b_3, a_4, c_6). 
\end{align*}
This gives $h_3 = t_{b_2} t_{b_3} t_{a_4} t_{c_6}$.

Let $h_4 = (h_0h_3)h_0(h_0h_3)^{-1}$, and hence $h_4$ is in $G$ by $h_3 \in G$. 
Then, we see that 
\begin{align*}
h_4 = (h_0h_3)t_{b_1}t_{b_2}t_{a_3}t_{c_5}(h_0h_3)^{-1} = t_{h_0h_3(b_1)}t_{h_0h_3(b_2)}t_{h_0h_3(a_3)}t_{h_0h_3(c_5)}. 
\end{align*}
Here, $b_1, b_2, a_4, c_5, c_6$ are disjoint from each other and $a_3, b_3$, and $a_3$ intersects $b_3$ transversely at exactly one point (see Figure~\ref{curves} (5)). 
This gives
\begin{align*}
&h_0h_3 = t_{b_1} t_{b_2} t_{a_3} t_{c_5} t_{b_2} t_{b_3} t_{a_4} t_{c_6} = t_{a_3} t_{b_3} \cdot t_{b_1} t_{b_2}^2 t_{c_5} t_{a_4} t_{c_6}, \\ 
&h_0h_3(b_1, b_2, a_3, c_5) = (b_1, b_2, t_{a_3} t_{b_3}(a_3), c_5) = (b_1, b_2, b_3, c_5), 
\end{align*}
and hence we have $h_4 = t_{b_1} t_{b_2} t_{b_3} t_{c_5}$. 
Since $b_1,b_2,c_5$ are disjoint from each other and $a_3,b_3$ (see Figure~\ref{curves} (5)), we obtain $h_0 h_4^{-1} =  t_{b_1}t_{b_2}t_{a_3}t_{c_5} t_{c_5}^{-1} t_{b_3}^{-1} t_{b_2}^{-1} t_{b_1}^{-1} = t_{a_3} t_{b_3}^{-1}$, and hence we get $t_{a_3} t_{b_3}^{-1} = h_0h_4^{-1} \in G$. 
We note that $\rho_n^2(a_3,b_3) = rt_{a_1}^n rt_{a_1}^n(a_3,b_3) = (a_5,b_5)$ since $a_1$ is disjoint from  $a_3,b_3,a_4,b_4$. 
Therefore, we get $t_{a_5} t_{b_5}^{-1} = t_{\rho_n^2(a_3)} t_{\rho_n^2(b_3)}^{-1} = \rho_n^2 t_{a_3} t_{b_3}^{-1} \rho_n^{-2} \in G$.

Let $h_5 = (h_4 t_{b_5} t_{a_5}^{-1}) h_4 (h_4 t_{b_5} t_{a_5}^{-1})^{-1}$, and hence $h_5$ is in $G$ by $h_4, t_{a_5} t_{b_5}^{-1} \in G$. 
Then, by $h_4 = t_{b_1} t_{b_2} t_{b_3} t_{c_5}$, we see that 
\begin{align*}
h_5 = t_{h_4 t_{b_5} t_{a_5}^{-1}(b_1)}t_{h_4 t_{b_5} t_{a_5}^{-1}(b_2)}t_{h_4 t_{b_5} t_{a_5}^{-1}(b_3)} t_{h_4 t_{b_5} t_{a_5}^{-1}(c_5)}.  
\end{align*}
Here, from Figure~\ref{curves} (6) we see that
\begin{itemize}
\item $b_1, b_2, b_3$ are disjoint from each other and $a_5, b_5, c_5$, 
\item $c_5$ is disjoint from $a_5$, and 
\item $c_5$ intersects $b_5$ transversely at exactly one point. 
\end{itemize}
These give 
\begin{align*}
&h_4 t_{b_5} t_{a_5}^{-1} = t_{b_1} t_{b_2} t_{b_3} t_{c_5} t_{b_5} t_{a_5}^{-1} = t_{c_5} t_{b_5} t_{a_5}^{-1} t_{b_1} t_{b_2} t_{b_3}, \ \mathrm{and} \\
&h_4 t_{b_5} t_{a_5}^{-1} (b_1, b_2, b_3, c_5) = (b_1, b_2, b_3, t_{c_5} t_{b_5} (c_5)) = (b_1, b_2, b_3, b_5), 
\end{align*}
and hence we have $h_5 = t_{b_1} t_{b_2} t_{b_3} t_{b_5}$. 
Since $b_1,b_2,b_3$ are disjoint from $b_5,c_5$ (see Figure~\ref{curves} (6)), we obtain $h_5h_4^{-1} = t_{b_1} t_{b_2} t_{b_3} t_{b_5} \cdot t_{c_5}^{-1} t_{b_3}^{-1} t_{b_2}^{-1} t_{b_1}^{-1} = t_{b_5}t_{c_5}^{-1}$,  and hence we get $t_{b_5} t_{c_5}^{-1} = h_5h_4^{-1} \in G$.

Let $h_6 = (h_2 t_{b_5} t_{a_5}^{-1}) h_2 (h_2 t_{b_5} t_{a_5}^{-1})^{-1}$, and hence $h_6$ is in $G$ by $h_2, t_{a_5} t_{b_5}^{-1} \in G$. 
Then, by $h_2 = t_{b_g} t_{t_{a_1}^{-n}(b_1)} t_{a_2} t_{c_4}$, we see that
\begin{align*}
h_6 = t_{h_2 t_{b_5} t_{a_5}^{-1}(b_g)}t_{h_2 t_{b_5} t_{a_5}^{-1}(t_{a_1}^{-n}(b_1))}t_{h_2 t_{b_5} t_{a_5}^{-1}(a_2)}t_{h_2 t_{b_5} t_{a_5}^{-1}(c_4)}. 
\end{align*}
Here, from Figure~\ref{curves} (7) we see that 
\begin{itemize}
\item $b_g, t_{a_1}^{-n}(b_1), a_2$ are disjoint from each other and $c_4, a_5, b_5$, 
\item $c_4$ is disjoint from $a_5$, and
\item $b_5$ intersects $c_4$ transversely at exactly one point. 
\end{itemize}
These give 
\begin{align*}
&h_2 t_{b_5} t_{a_5}^{-1} = t_{b_g} t_{t_{a_1}^{-n}(b_1)} t_{a_2} t_{c_4} t_{b_5} t_{a_5}^{-1}, \ \mathrm{and} \\
&h_2 t_{b_5} t_{a_5}^{-1} (b_g, t_{a_1}^{-n}(b_1), a_2, c_4) = (b_g, t_{a_1}^{-n}(b_1), a_2, t_{c_4}t_{b_5}(c_4)) =  (b_g, t_{a_1}^{-n}(b_1), a_2, b_5), 
\end{align*}
and hence we have $h_6 = t_{b_g} t_{t_{a_1}^{-n}(b_1)} t_{a_2} t_{b_5}$. 
Since $b_g,t_{a_1}^{-n}(b_1),a_2$ are disjoint from $c_4,b_5$ (see Figure~\ref{curves} (7)), we obtain $h_6h_2^{-1} =  t_{b_g} t_{t_{a_1}^{-n}(b_1)} t_{a_2} t_{b_5} \cdot  t_{c_4}^{-1} t_{a_2}^{-1} t_{t_{a_1}^{-n}(b_1)}^{-1} t_{b_g}^{-1} = t_{b_5}t_{c_4}^{-1}$, and hence we get $t_{b_5} t_{c_4}^{-1} = h_6h_2^{-1} \in G$.

Summarizing, we have $t_{a_5}t_{b_5}^{-1}, t_{b_5}t_{c_5}^{-1}, t_{b_5}t_{c_4}^{-1} \in G$. 
We see that $\rho_n(a_5,b_5,c_4,c_5) = rt_{a_1}^n(a_5,b_5,c_4,c_5) = (a_6,b_6,c_5,c_6)$ since $a_1$ is disjoint from $a_5,b_5,c_4,c_5$. 
This gives 
\begin{align*}
&t_{a_6}t_{b_6}^{-1} = t_{\rho_n(a_5)} t_{\rho_n(b_5)}^{-1} = \rho_n t_{a_5}t_{b_5}^{-1} \rho_n^{-1} \in G, \\
&t_{b_6}t_{c_5}^{-1} = t_{\rho_n(b_5)} t_{\rho_n(c_4)}^{-1} = \rho_n t_{b_6}t_{c_4}^{-1} \rho_n^{-1} \in G \ \mathrm{and} \\
&t_{b_6}t_{c_6}^{-1} = t_{\rho_n(b_5)} t_{\rho_n(c_5)}^{-1} = \rho_n t_{b_5}t_{c_5}^{-1} \rho_n^{-1} \in G. 
\end{align*}
Therefore, we get
\begin{align*}
&t_{c_5} t_{c_6}^{-1} = (t_{b_6}t_{c_5}^{-1})^{-1} \cdot t_{b_6} t_{c_6}^{-1} \in G, \\
&t_{b_5} t_{b_6}^{-1} = t_{b_5}t_{c_5}^{-1} \cdot (t_{b_6} t_{c_5}^{-1})^{-1} \in G \ \mathrm{and} \\
&t_{a_5} t_{a_6}^{-1} = t_{a_5}t_{b_5}^{-1} \cdot t_{b_5} t_{b_6}^{-1} \cdot (t_{a_6} t_{b_6}^{-1})^{-1} \in G. 
\end{align*}
We note that $\rho_n(a_i,b_i,c_i)=rt_{a_1}^n(a_i,b_i,c_i) = (a_{i+1},b_{i+1},c_{i+1})$ for $2\leq i\leq g-1$ since $a_1$ is disjoint from $a_i,b_i,c_i$. 
From Lemma~\ref{lem:7}, we see that the elements $t_{a_5}, t_{b_5} , t_{c_5}$ are in $G$ (by letting $k = 5$ and $f = \rho_n$). 
In particular, since $\rho_n(a_1) = rt_{a_1}^n(a_1) = r(a_1) = a_2$, we have $\rho_n^{-4}(a_5) = a_1$, and hence we get $t_{a_1} = t_{\rho_n^{-4}(a_5)} = \rho_n^{-4} t_{a_5} \rho_n \in G$ and $r = r t_{a_1}^{n} t_{a_1}^{-n} = \rho_n t_{a_1}^{-n} \in G$. 
By $r^j(a_5,b_5,c_5) = (a_{j+5},b_{j+5},c_{j+5})$ for $j=-4,-3,\ldots,g-5$ and $t_{a_5}, t_{b_5}, t_{c_5} \in G$, we see that $t_{a_{j+5}} = t_{r^j(a_5)} = r^j t_{a_5} r^{-j} \in G$, $t_{b_{j+5}} = t_{r^j(b_5)} = r^j t_{b_5} r^{-j} \in G$ and $t_{c_{j+5}} = t_{r^j(c_5)} = r^j t_{c_5} r^{-j} \in G$. 
By Theorem~\ref{Lickorishthm}, we have $G=\M$, which completes the proof. 
\end{proof}
\begin{figure}[hbt]
  \centering
       \includegraphics[scale=.60]{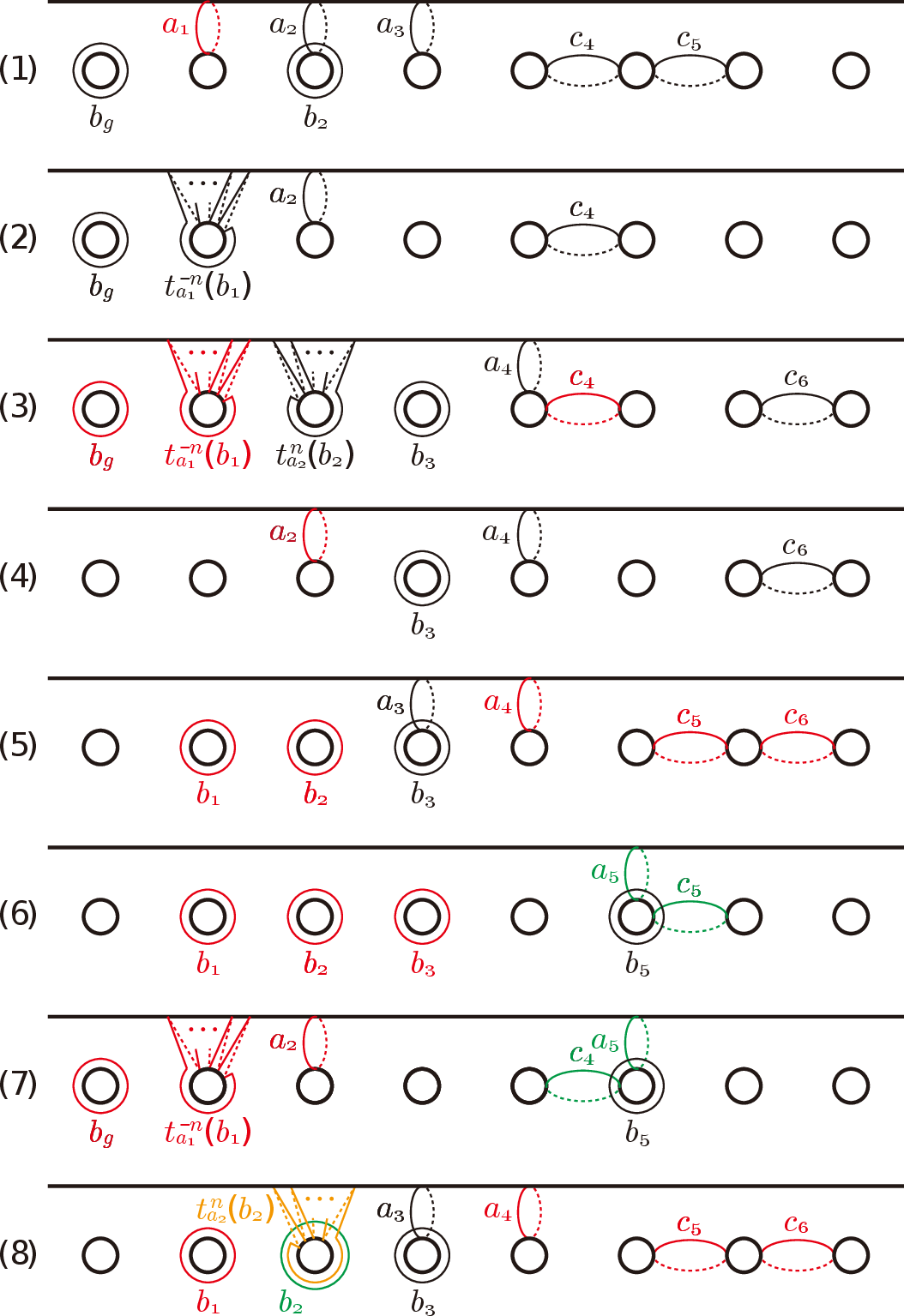}
       \caption{The curves appearing in the proof of Theorem~\ref{thm:2} and Proposition~\ref{prop:3}.}
       \label{curves}
  \end{figure}

\begin{prop}\label{prop:3}
The pair $(h_0, \rho_m)$ is not conjugate to $(h_0, \rho_n)$ if $m \neq n$ and $m,n>0$, where $\rho_n = rt_{a_1}^n$ and $h_0 = t_{b_1}t_{b_2}t_{a_3}t_{c_5}$.  
\end{prop}
\begin{proof}
Let us consider the commutator $[h_0,\rho_n]$. 
Then, we see that 
\begin{align*}
[h_0,\rho_n] = t_{b_1}t_{b_2}t_{a_3}t_{c_5} \rho_n (t_{b_1}t_{b_2}t_{a_3}t_{c_5})^{-1} \rho_n^{-1} = t_{b_1}t_{b_2}t_{a_3}t_{c_5} (t_{\rho_n(b_1)}t_{\rho_n(b_2)}t_{\rho_n(a_3)}t_{\rho_n(c_5)})^{-1}. 
\end{align*}
Here, since $a_1$ is disjoint from $b_2,a_3,c_5$ (see Figure~\ref{curves} (1)), we have
\begin{align*}
\rho_n(b_1,b_2,a_3,c_5) &= r t_{a_1}^n(b_1,b_2,a_3,c_5) = r(t_{a_1}^n(b_1),b_2,a_3,c_5) = (t_{a_2}^n(b_2),b_3,a_4,c_6). 
\end{align*}
Moreover, since $b_1, a_4, c_5, c_6$ (resp. $b_2, t_{a_2}^n(b_2)$) are disjoint from $b_2, t_{a_2}^n(b_2), a_3, b_3$ (resp. $a_3, b_3$) (see Figure~\ref{curves} (8)), we obtain 
\begin{align*}
[h_0,\rho_n] = t_{b_1}t_{b_2}t_{a_3}t_{c_5} (t_{t_{a_2}^n(b_2)}t_{b_3}t_{a_4}t_{c_6})^{-1} = t_{b_1} t_{a_4}^{-1} t_{c_5} t_{c_6}^{-1} \cdot t_{b_2} t_{t_{a_2}^n(b_2)}^{-1} \cdot t_{a_3}t_{b_3}^{-1}. 
\end{align*}
If two elements of  $\M$ are conjugate to each other, then the sets of eigenvalues of their 
actions on $H_1(\Sigma_g;\mathbb{R})$ are the same. 
Therefore, in order to show that $[h_0,\rho_m]$ is not conjugate to $[h_0,\rho_n]^{\pm 1}$ 
when $m \not= n$ and $m,n>0$, we consider the set of eigenvalues of their actions on 
$H_1(\Sigma_g;\mathbb{R})$. 
Let $m_i$ (resp. $l_i$) be an element of $H_1(\Sigma_g ; \mathbb{R})$ represented by 
the oriented curve $a_i$ (resp. $b_i$) in Figure ~\ref{rotation}. 
Let $R_i$ ($i=1,2,\dots,6)$ be the subspace of $H_1(\Sigma_g;\mathbb{R})$ defined by 
\begin{align*}
R_j &= \mathbb{R}m_j \oplus \mathbb{R}l_j \ \ (j=1,2,3,4), \\
R_5 &= \mathbb{R}m_5 \oplus \mathbb{R}l_5 \oplus  \mathbb{R}m_6 \oplus \mathbb{R}l_6 \oplus 
 \mathbb{R}m_7 \oplus \mathbb{R}l_7, \\
R_6 &=  \mathbb{R}m_8 \oplus \mathbb{R}l_8 \oplus \cdots 
\oplus \mathbb{R}m_g \oplus \mathbb{R}l_g .
\end{align*} 
Then $H_1(\Sigma_g;\mathbb{R})$ is a direct sum of $R_i$'s and 
the actions of $[h_0,\rho_m]$ and $[h_0,\rho_n]^{\pm 1}$ preserve these subspaces. 
It is easy to see that the eigenvalues of the action of 
$[h_0,\rho_m]$ and $[h_0,\rho_n]^{\pm 1}$ 
on $R_i$ ($ i \not= 2$) are the same. 
On the other hand, the eigenvalues of the actions of $[h_0,\rho_m]$ on $R_2$ are 
$\displaystyle{\frac{m^2 + 2 \pm \sqrt{(m^2+2)^2 + 4}}{2}}$ and 
that of $[h_0,\rho_n]^{\pm 1}$ are 
$\displaystyle{\frac{n^2 + 2 \pm \sqrt{(n^2+2)^2 + 4}}{2}}$. 
Since the function $f(a) = a + \sqrt{a^2+4}$ is strictly increasing for $a > 0$, 
we see that the set of eigenvalues of the action of 
$[h_0,\rho_m]$ and $[h_0,\rho_n]^{\pm 1}$ on $H_1(\Sigma_g;\mathbb{R})$ are 
not the same when $m \not=n$ and $m,n>0$.  
Therefore, if $m\neq n$ and $m,n>0$, then $[h_0,\rho_m]$ is not conjugate to $[h_0,\rho_n]$ and $[h_0,\rho_n]^{-1}$, which completes the proof. 
\end{proof}

\begin{proof}[Proof of Theorem~\ref{thm:1}]
Theorem~\ref{thm:1} immediately follows from Theorem~\ref{thm:2}, Proposition~\ref{prop:3} and Lemma~\ref{commutator}. 
\end{proof}

\begin{proof}[Proof of Proposition~\ref{prop:100}]
The group $\mathrm{SL}(2,\mathbb{Z}) = \langle x,y\mid x^2y^{-3} = x^4 =1 \rangle$ is a free product of two cyclic groups $G_1 = \langle x \mid x^4=1\rangle \cong \mathbb{Z}_4$ and $G_2 = \langle y\mid y^6=1\rangle \cong \mathbb{Z}_6$ with amalgamated cyclic subgroup $G_0 = \langle x^2\mid (x^2)^2=1\rangle = \langle y^3 \mid (y^3)^2=1\rangle \cong \mathbb{Z}_2$ (see, for example, \cite{Serre}). 
Since $G_0$ is the center of both $G_1$ and $G_2$ (and hence $G_0$ is a normal subgroup of both $G_1$ and $G_2$), it is a normal subgroup of $G$. 
By Lemma~\ref{lem:1000}, every generating pair $(x_1,x_2)$ of $\mathrm{SL}(2,\mathbb{Z})$ is Nielsen equivalent to a pair $(y_1,y_2)$, which is contained in the finite set $G_1 \times G_2$. 
This finishes the proof. 
\end{proof}

\begin{rem}
Here is an alternative proof of Proposition~\ref{prop:100}, suggested by Makoto Sakuma. 
There are also only finitely many Nielsen equivalence classes on generating pairs of $PSL(2,\mathbb{Z}) = \mathbb{Z}_2 \ast \mathbb{Z}_3$ by Theorem~\ref{thm:G}. 
Let $[\overline{A},\overline{B}]$ be the image of the Nielsen equivalence class on a generating pair of $SL(2,\mathbb{Z})$ under the natural projection $SL(2,\mathbb{Z}) \to PSL(2,\mathbb{Z}) = SL(2,\mathbb{Z})/\{I, -I\}$, and hence $[\overline{A},\overline{B}]$ is the Nielsen equivalence class on a generating pair $(\overline{A},\overline{B})$. 
Since then the inverse image of $[\overline{A},\overline{B}]$ is $\{[A,B], [A,-B], [-A,B], [-A,-B]\}$, it follows from the above fact that $SL(2,\mathbb{Z})$ also has finitely many Nielsen equivalence classes. 
\end{rem}

\begin{proof}[Proof of Theorem~\ref{thm:1000}]
By Theorem 1 (Detailed Version) (1) in \cite{Mc} for $g\geq 3$, the automorphism group of $\M$ is generated by 
the maps $\phi (f) = \rho f \rho^{-1}$, where $\rho$ is the isotopy class of a possibly orientation-reversing homeomorphism of $\Sigma_g$. 
Let $[h_0,\rho_n]$ be the commutator of $h_0$ and $\rho_n$, where $(h_0,\rho_n)$ is the generating pair of $\M$ given in Theorem~\ref{thm:2}. 
Since $[\phi(h_0),\phi(\rho_n)] = [\rho h_0 \rho^{-1}, \rho \rho_n \rho^{-1}] = \rho[h_0, \rho_n] \rho^{-1}$, 
the set of eigenvalues of the action of $[\phi(h_0),\phi(\rho_n)]$ on $H_1(\Sigma_g;\mathbb{R})$ is the same as that of $[h_0, \rho_n]$. 
Therefore, $[\phi(h_0),\phi(\rho_m)]$ is not conjugate to $[\phi(h_0),\phi(\rho_n)]$  if $m\neq n$. 
This means that $(h_0,\rho_m)$ is not T-equivalent to $(h_0,\rho_n)$ if $m\neq n$, which completes the proof. 
\end{proof}

\end{document}